%% file: main.tex
\documentclass[review,onefignum,onetabnum]{siamart190516}
\input{ex_shared_v2}

\ifpdf
\hypersetup{
  pdftitle={1dPoisson},
  pdfauthor={}
}
\fi

\usepackage{amsmath,rotating}

\usepackage{soul} 

\usepackage[utf8]{inputenc}
\usepackage{bm}
\usepackage{mathtools}
\usepackage{indentfirst}
\usepackage{subfigure}
\usepackage{tcolorbox}
\usepackage{color}
\usepackage[export]{adjustbox}
\usepackage{scalerel}
\usepackage{tikz}
\usepackage{geometry}
\usepackage{multicol}
\usetikzlibrary{decorations.pathreplacing}
\usepackage{graphicx}
\usepackage{subcaption}
\usepackage{amsmath,amssymb}

\usepackage{algorithmic}
\Crefname{ALC@unique}{Line}{Lines}

\newcommand{\R}{\mathbb{R}}

\usepackage{graphicx}
\usepackage{diagbox}
\usepackage{pifont}
\newcommand{\vertiii}[1]{{\left\vert\kern-0.25ex\left\vert\kern-0.25ex\left\vert #1 
    \right\vert\kern-0.25ex\right\vert\kern-0.25ex\right\vert}}

\newcommand{\bSigma}{\mbox{\boldmath${\Sigma}$}}

\newcommand{\bxi}{\mbox{\boldmath$\xi$}}

\newtheorem{thm}{\bf Theorem}[section]
\newtheorem{lem}[thm]{Lemma}

\usepackage{enumitem}
\setlist[itemize]{left=16pt} 

\def\b1{{\bf 1}}
\def\bb{{\bf b}}
\def\bB{{\bf B}}
\def\bc{{\bf c}}
\def\bd{{\bf d}}
\def\bD{{\bf D}}
\def\bff{{\bf f}}
\def\bg{{\bf g}}
\def\bq{{\bf q}}

\def\bH{{\bf H}}

\def\cM{{\cal M}}

\usepackage{cleveref}
\crefname{rem}{Remark}{remarks}
\Crefname{rem}{Remark}{Remarks}

\begin{document}



%

\maketitle

\begin{abstract} 
This paper studies the shallow Ritz method for solving the one-dimensional diffusion problem. It is shown that the shallow Ritz method improves the order of approximation dramatically for non-smooth problems. To realize this optimal or nearly optimal order of the shallow Ritz approximation, we develop a damped block Newton (dBN) method that alternates between updates of the linear and non-linear parameters. Per each iteration, the linear and the non-linear parameters are updated by exact inversion and one step of a modified, damped Newton method applied to a reduced non-linear system, respectively. The computational cost of each dBN iteration is $\mathcal{O}(n)$. 

Starting with the non-linear parameters as a uniform partition of the interval, numerical experiments show that the dBN is capable of efficiently moving mesh points to nearly optimal locations. To improve the efficiency of the dBN further, we propose an adaptive damped block Newton (AdBN) method by combining the dBN with the adaptive neuron enhancement (ANE) method \cite{Liu_Cai_Chen_2022}.

\end{abstract}



\section{Introduction}

In the past decade, the use of neural networks (NNs) has surged in popularity, finding applications in artificial intelligence, natural language processing, image recognition, and various other domains within machine learning. Consequently, the use of NNs has extended to diverse fields, including numerically solving partial differential equations (PDEs) \cite{BERG201828,  CAI2020109707, Do19, E_Yu_2018, RAISSI2019686, sirignano2018dgm}. The idea of using NNs to solve PDEs may be traced back to the 90s \cite{Dissanayake94, Lagaris00, Lagaris98} based on a simplified Bramble-Schatz least-squares (BSLS) formulation \cite{bs1} in 1970. 

This simplified version of the BSLS applies the $L^2$ norm least-squares (LS) principle to the strong form of the underlying PDE and boundary conditions. For non-smooth problems, the BSLS formulation is {\it not} equivalent to the original problem because the critical interface condition is ignored (see, e.g., \cref{interface}). For various NN methods based on equivalent LS formulations, see \cite{CAI2020109707} for the convection-diffusion-reaction problem, \cite{Cai2021linear} for the advection-reaction problem, and \cite{CAI2023115298} for the scalar non-linear hyperbolic conservation laws. 

For a class of self-adjoint problems with a natural minimization principle, one may employ the Ritz formulation instead of manufactured LS formulations. Recently, the deep Ritz method was introduced in \cite{E_Yu_2018} for the Poisson equation and uses deep neural networks (DNNs) as approximating functions. Since the Ritz formulation automatically enforces the interface condition weakly, the deep Ritz method is applicable to non-smooth problems such as elliptic interface problems. 

In one dimension, the shallow ReLU NN is sufficient to accurately approximate non-smooth functions because it 
produces the same class of approximating functions as the free knots splines (FKS). The FKS was introduced in the 1960s and studied by many researchers \cite{Barrow_Chui_Smith_Ward_1978,deBoorRice1968, jupp1978, Schumaker}. Spline approximations to non-smooth functions can be improved dramatically with free knots \cite{Burchard74} (see \Cref{Section2}). Nevertheless, determining optimal knot locations (the non-linear parameters of a shallow NN) becomes a complicated, computationally intensive non-convex optimization problem. Although FKS was a subject of many research articles on various optimization methods such as the DFP method \cite{Davidon59, Fletcher64}, the Gauss-Newton method \cite{Golub1973}, a method moving each knot locally \cite{deBoorRice1968, LOACH_WATHEN_1991}, etc. (see, e.g., \cite{jupp1978, Osborne70} and references therein), to the best of our knowledge, there is no optimization method, including those commonly used in scientific machine learning, that can move mesh points (the non-linear parameters) to nearly optimal locations efficiently and effectively. 


The main purpose of this paper is to address this challenging issue on the non-convex optimization problem arising from the shallow Ritz approximation of one-dimensional diffusion problems. The shallow ReLU NN has two types of parameters: linear and non-linear. They are corresponding to the weights and biases of the output and hidden layers, respectively. Moreover, the non-linear ones are mesh points of the continuous piecewise linear NN approximation, and the linear ones are the coefficients of the linear combination of one and neuron functions. Due to their geometrical meanings, it is then natural and efficient to employ the commonly used outer-inner iterative method by alternating between updates of the linear and non-linear parameters.

The optimality condition for the linear parameters leads to a system of linear equations whose coefficient matrix and right-hand side vector depend on the non-linear parameters. Due to the non-local support of neuron functions, this stiffness matrix is dense and ill-conditioned. Specifically, we show that its condition number is bounded by $\mathcal{O}\left(n\,h^{-1}_{\text{min}}\right)$ (see \cref{l:Cond}), where $n$ is the number of neurons and $h_{\text{min}}$ is the smallest distance between two neighboring mesh points. To overcome this difficulty in the linear solver, we find that the inverse of the stiffness matrix is a tridiagonal matrix (see \Cref{coefficientmatrixInverse}). 





The optimality condition for the non-linear parameters is a nearly decoupled system of non-linear algebraic equations with a linear coupling through the penalized boundary term. This algebraic structure implies that the Newton method could be the best choice for solving this non-linear system. However, the gradient of the energy functional with respect to the non-linear parameters is not differentiable at the physical interface for the elliptic interface problem, and the Hessian matrix may not be invertible due to either some vanishing linear parameters or vanishing main parts of some diagonals of the Hessian matrix. In both cases, we identify that the corresponding non-linear parameters are either unneeded or can be fixed. Hence, they may be removed from the non-linear system. The reduced non-linear system is then differentiable and its Hessian is invertible. 

Combining the exact inversion for the linear parameters and one step damped Newton method for the reduced non-linear system, we introduce the damped block Newton (dBN) method. Computational cost of each dBN iteration is $\mathcal{O}(n)$. Starting with uniformly partitioned non-linear parameters of the interval, numerical experiments show that the dBN is capable of efficiently moving mesh points to nearly optimal locations. To improve the efficiency of the dBN further, we propose an adaptive damped block Newton (AdBN) method by combining the dBN with the adaptive neuron enhancement (ANE) method \cite{Liu_Cai_Chen_2022}. Numerical examples demonstrate the ability of AdBN not only to move mesh points quickly and efficiently but also to achieve a nearly optimal order of approximation. Both the dBN and AdBN outperform BFGS in terms of cost and accuracy.

{\bf Related work.} In the context of the shallow NN for multi-dimension, the active neuron least squares (ANLS) method was recently introduced in \cite{Ainsworth2020PlateauPI, Ainsworth2022} to efficiently update the non-linear parameters. By utilizing both the quadratic structure of the functional and the NN structure, the structure-guided Gauss-Newton (SgGN) method was newly proposed in \cite{Cai24GN} for solving the non-linear least-squares problem. For several one- and two-dimensional least-squares test problems which are difficult for the commonly used training algorithms in machine learning such as BFGS and Adam, the SgGN shows superior convergence. However, the mass matrix for the linear parameters and the layer Gauss-Newton matrix are ill-conditioned even though they are symmetric and positive definite. 

The paper is structured as follows. \Cref{Section2} describes FKS and shallow ReLU NNs as well as their equivalence. The shallow Ritz method for the diffusion equation is presented and analyzed in \Cref{Section3}. Optimality conditions of the shallow Ritz discretization are derived and the condition number of the stiffness matrix is estimated in \Cref{SAE}. The damped block Newton (dBN) method and its adaptive version are introduced in Sections~\ref{Section4} and \ref{section adaptivity}, respectively. Finally, numerical results are presented in \Cref{Section numerical exp}.

\section{Free-Knot Splines and Shallow Neural Networks}\label{Section2}

This section describes continuous linear free-knot splines (FKS) \cite{Schumaker} and shallow neural networks in one dimension as well as their equivalence.

Set $b_{-1}=b_0=0$ and $b_{n+2}=b_{n+1}=1$. Let $\phi_i(x)$ be the standard hat basis function with support on $(b_{i-1},b_{i+1})$ given by 
\begin{equation}\label{phi}
    \phi_i(x)=\left\{ \begin{array}{ll}
(x-b_{i-1})/(b_i-b_{i-1}), & x\in (b_{i-1}, b_i),\\[2mm]
(b_{i+1}-x)/(b_{i+1}-b_{i}), & x\in (b_{i}, b_{i+1}),\\[2mm]
0, & \mbox{otherwise}. \end{array}
\right.
\end{equation} 
Then the linear FKS on the interval $I=[0,1]$ is defined as 
\begin{equation}\label{free}
    \mathcal{S}_{n}(I)=\left\{\sum_{i=0}^{n+1} c_i\phi_i(x):\, c_i, b_i\in \R \text{ and } 0<b_1<\ldots<b_n<1\right\}.
\end{equation}
The $\{b_i\}_{i=1}^n$ are referred to as the knots, or breaking points, and are the parameters to be determined. 

The FKS was introduced in the late 1960s and has been extensively studied since then. By adjusting the breaking points, the FKS gains substantially in approximation order for non-smooth functions. For example, let $0<\alpha<1$, then the order of the best approximation to $x^\alpha$ on $I$ is $\mathcal{O}(n^{-\alpha})$ when using linear finite elements on the uniform mesh. Whereas for the FKS, the order is $\mathcal{O}(n^{-1})$ (see, e.g., \cite{daubechies, Schumaker}). Despite this remarkable approximation capability of FKS, numerical analysts have largely moved away from this method due to two main challenges: (1) no successful extension of FKS to two or higher dimensions has been achieved, and (2) determining the optimal placement of $\{b_i\}_{i=1}^n$ has remained a difficult and unresolved problem.

Denote by $\sigma(x)=\max\{0, x\}$ the ReLU activation function. Set $b_0=0$ and $b_{n+1}=1$. Then the collection of the shallow neural network (NN) functions introduced by Rosenblatt in 1958  is given by
\begin{eqnarray}\label{NNs}
    {\cal M}_n(I)=\left\{c_{-1}+ \sum_{i=0}^{n}c_i\sigma(x-b_i) \, :\, c_i, b_i\in \R \text{ and } 0<b_1<\ldots<b_n<1\right\}, 
\end{eqnarray}
where $\{b_i\}_{i=1}^n$ are the interior ``mesh'' points to be determined and the typically present weights \{$\omega_i\}_{i=0}^n$ are normalized (i.e., $\omega_i=1$). Clearly, any function in ${\cal M}_n(I)$ is continuous piecewise linear with respect to the partition of the interval $I$ by the breaking points $\{b_i\}_{i=1}^n$ and the boundary $\partial I=\{0,1\}= \{b_0,b_{n+1}\}$. Therefore, the shallow ReLU NN is equivalent to the FKS, i.e.  $\mathcal{S}_n(I) = {\cal M}_n(I)$, and can be extended to two or higher dimensions directly. This resolves the first challenge of the FKS due to the change of the basis functions from local support to global support. However, the price for this gain is complexity and ill-condition of the resulting algebraic structure. We will address this second fundamental challenge of FKS in \cref{Section4}.

\section{Shallow Ritz Method}\label{Section3}

Consider the following one-dimensional diffusion equation
\begin{equation}\label{pde}
    \left\{\begin{array}{lr}
        -(a(x)u^{\prime}(x))^{\prime}=f(x), & x\in I=(0,1),\\[2mm]
        u(0)= \alpha,  \quad u(1)=\beta, &
    \end{array}\right.
\end{equation}
where $f(x)$ is a given real-valued function defined on $I$. Assume that the diffusion coefficient $a(x)$ is bounded below by a positive constant $\mu > 0$ almost everywhere on $I$. When $a(x)$ is a piecewise constant, \cref{pde} is referred to as an elliptic interface problem. Clearly, the strong form in \cref{pde} is invalid on the physical interface $\Gamma$ where $a$ is discontinuous. For completion, an additional {\it interface condition} \cref{interface} must be supplemented on the interface $\Gamma$:
\begin{equation}\label{interface}
(au^\prime)^+ \big|_{\Gamma} -(au^\prime)^-\big|_{\Gamma}=0. 
\end{equation} 

Due to the non-local support of neurons, one of the Dirichlet boundary conditions is enforced strongly and the other can be enforced either algebraically (see \cref{Appendix} for the formulation) or weakly by penalizing the energy functional. We opt for the latter; the modified Ritz formulation of problem \cref{pde} is to find $u \in H^{1}(I)\cap \{u(0)= \alpha\}$
such that

\begin{equation}\label{energy_functional}
    J(u) = \min_{v \in H^{1}(I)\cap \{v(0)= \alpha\}}J(v),
\end{equation}
where the modified energy functional is given by 
\begin{equation*}\label{modifiedRitzFxnl}
    J(v) = \frac{1}{2}\int_{0}^1a(x)(v^{\prime}(x))^2dx  - \int_{0}^1f(x)v(x)dx + \frac{\gamma}{2}(v(1) - \beta)^2.
\end{equation*}
Here, $\gamma > 0$ is a penalization constant. Then the Ritz neural network approximation is to find $u_n\in {\cal M}_n(I) \cap \{u_{n}(0) = \alpha\}$ such that
\begin{equation}\label{min}
    J(u_n)=\min_{v\in\cM_n(I) \cap \{v(0) = \alpha\}}J(v).
\end{equation}

Below we estimate the error bound for the solution $u_n$ to \cref{min}. To this end, define the bilinear and linear forms 
by
\begin{equation}\label{af}
    a(u, v) := \int_0^1a(x)u^{\prime}(x)v^{\prime}(x)dx + \gamma u(1)v(1) \quad\mbox{and}\quad
    f(v):=\int_{0}^1f(x)v(x)dx+\gamma\beta\, v(1),
\end{equation}
respectively, and denote the induced norm of the bilinear form by $\|v\|_{a}^2 = a(v,v)$. Then the modified energy function is given by
\begin{equation}\label{J}
     J(v) = \frac{1}{2}a(v,v)-f(v)+\frac{1}{2}\gamma\beta^2.
\end{equation}

\begin{lemma}\label{l:error_estimate1}
    Let $u$ and $u_n$ be the solutions of problems \cref{energy_functional} and \cref{min}, respectively. Then 
    \begin{equation}\label{error}
        \|u-u_n\|_{a} \leq \sqrt{3}\inf_{v\in\cM_n(I) \cap \{v(0) = \alpha\}}\|u-v\|_{a} + 2\sqrt{2}\, \big|a(1)u^{\prime}(1)\big|\, \gamma^{-1/2}.
    \end{equation}
\end{lemma}

\begin{proof}
It is easy to see that the solution $u$ of \cref{pde} satisfies
\begin{equation}\label{vp}
    a(u,v) +\alpha a(0)u'(0)-a(1)u'(1)v(1)=f(v)
\end{equation}
for any $v \in H^{1}(I) \cap \{v(0) = \alpha\}$. Together with \cref{J}, we have
\[
a(v - u, v - u)= 2(J(v) - J(u))+ 2a(1)u^{\prime}(1)(u(1) - v(1)),
\]
which, combining with the inequality that $2c\,d\leq 2\gamma^{-1}c^2 + \frac{\gamma}{2}d^2$, implies
\[
 \|v-u\|^2_{a} - 4\big|a(1)u^{\prime}(1)\big|^2\gamma^{-1} \leq 4(J(v) - J(u) )\leq 3 \|v-u\|^2_{a} + 4\big|a(1)u^{\prime}(1)\big|^2\gamma^{-1}.
\]
Now, together with the fact that $J(u_n)\leq J(v)$ for any $v \in \cM_n(I) \cap \{v(0) = \alpha\}$, we have 
\[
\|u_n-u\|^2_{a} \leq 4 (J(v) - J(u)) + 4\big|a(1)u^{\prime}(1)\big|^2\gamma^{-1} \leq 3 \|v - u\|^2_{a} + 8 \big|a(1)u^{\prime}(1)\big|^2\gamma^{-1}.
\]
This implies \cref{error} and proves the lemma.
\end{proof}

Since $\cM_n(I)=\mathcal{S}_{n}(I)$, it is reasonable to assume that there exists a constant $C(u)$ depending on $u$ such that 
\begin{equation}\label{app}
    \inf_{v\in\cM_n(I)}\|a^{1/2}(u - v)^{\prime}\|_{L^2(I)} \leq C(u)\, n^{-1},
\end{equation}
provided that $u$ has certain smoothness. Obviously, \cref{app} is valid for $u\in H^2(I)=W^{2,2}(I)$ even when the breaking points are fixed and form a quasi-uniform partition of the interval $I$. It is expected that \cref{app} is also valid for $u\in W^{2,1}(I)$ when the breaking points are free \cite{Schumaker}. 

\begin{proposition}\label{p:final_error_est}
    Let $u$ and $u_n$ be the solutions of the problem \cref{energy_functional} and \cref{min}, respectively. Assume that $a\in L^\infty(I)$, then there exists a constant $C$ depending on $u$ such that 
    \begin{equation}\label{eb}
        \|u - u_n\|_a \leq C \left(n^{-1}+\gamma^{-1/2}\right).
    \end{equation}
\end{proposition}
    
\begin{proof}
The proposition is a direct consequence of \cref{l:error_estimate1} and \cref{app}.
\end{proof}

We close this section by introducing a concept on optimal breaking points.

\begin{definition}\label{bps}
The breaking points $\{\hat{b}_i\}_{i=1}^n$ of a NN function $\hat{u}_n\in \cM_n(I)$ are said to be optimal if $\hat{u}_n\in \cM_n(I)$ satisfies the error bound in \cref{eb}.
\end{definition}

\section{Optimality Conditions}\label{SAE}
This section derives optimality conditions of \cref{min} and estimates the condition number of the stiffness matrix. 

To this end, let $u_n(x)=\alpha + \sum\limits_{i=0}^{n}c_i\sigma(x - b_i)$ be a solution of \cref{min}. Denote by $\bc=(c_0,\ldots, c_{n})^T$ and $\bb=\left(b_1,\ldots, b_n\right)^T$ the respective linear and non-linear parameters. Set 
\[
\bSigma=\bSigma(x) = (\sigma(x - b_0), \ldots, \sigma(x - b_n))^T \quad\mbox{and}\quad {\bf H}={\bf H}(x) = (H(x - b_0), \ldots, H(x - b_n))^T,
\]
where $H(t)= \sigma^{\prime}(t)$ is the Heaviside step function given by 
\begin{equation*}
H(t)= \sigma^{\prime}(t)=
    \begin{cases}
        1, & \quad t>0,\\[2mm]
        0, & \quad t < 0.
    \end{cases}
\end{equation*}
Clearly, we have the following identities
\begin{equation}\label{4.1}
u_n(x) = \alpha + \bSigma^T \bc, \quad u^\prime_n(x) = {\bf H}^T \bc, \quad
\nabla_{\bc}u_n(x)= \bSigma,
\quad\mbox{and}\quad
\nabla_{\bc}u^\prime_n(x)= {\bf H}, 
\end{equation}
where $\nabla_{\bc}$ and $\nabla_{\bb}$ are the gradient operators with respect to the $\bc$ and $\bb$, respectively. 

By \cref{4.1}, we have 
\begin{eqnarray*}
    \nabla_{\bc} J(u_n) 
        & = &  \int_0^1 a(x)u_n^{\prime}(x)\nabla_{\bc}u_n^{\prime}(x)dx  - \int_0^1 f(x)\nabla_{\bc}u_n(x)dx + \gamma(u_n(1) - \beta) \nabla_{\bc}u_n(1)\\[2mm]
        & = &\left[ \int_0^1 a(x) \bH(x)\bH(x)^T dx\right] \bc - \int_0^1 f(x)\bSigma(x)dx +  \gamma\left(\bd^T\bc +\alpha- \beta\right)\bd,
\end{eqnarray*}
where $\bd=\bSigma(1)$. Denote by 
\begin{equation*}
    A(\bb)=\int_0^1 a(x)\bH(x)\bH(x)^T dx
    \quad \mbox{and} \quad \bff(\bb)=\int_0^1 f(x)\bSigma(x)dx,
\end{equation*}
the stiffness matrix and right hand side vector, respectively, with components
\begin{equation*}
    a_{ij}(\bb) = \int_0^1{a(x)H(x - b_{i-1})H(x - b_{j-1})dx}\quad\mbox{and}\quad f_i = \int_0^1f(x)\sigma(x - b_{i-1})dx.
\end{equation*}  
Now, the optimality condition of \cref{min} with respect to the linear parameters $\bc$ is given by
\begin{equation}\label{linear_eq}
    {\bf 0}=\nabla_{\bc} J(u_n) = {\mathcal A} (\bb)\, \bc- \mathcal{F}(\bb),
\end{equation}
where ${\mathcal A} (\bb)=A(\bb)  + \gamma \bd \bd^T$ and $\mathcal{F}(\bb)= \bff(\bb)+\gamma(\beta-\alpha)\bd$.


Next, we derive the optimality condition of \cref{min} with respect to the non-linear parameters $\bb$.
For $j=1,\ldots, n$, let
\begin{equation}\label{gj}
q_j=\int_{b_j}^1 f(x)\,dx   - a(b_j) u^\prime_n(b_j), \quad\mbox{where}\quad u^\prime_n(b_j):= \frac{u_n'(b_j^+)+u_n'(b_j^-)}{2}=
\sum_{i=0}^{j-1} c_i + c_{j}/2.  
\end{equation}
Additionally, let 
\begin{equation*}
    \bq = (q_1,\ldots, q_n)^T,\quad \hat{\bc}=(c_1,\ldots, c_n),\quad \mbox{and}\quad  \b1=(1,\ldots, 1)^T.
\end{equation*}
\begin{lemma}\label{gradB=0}
The optimality condition of \cref{min} with respect to $\bb$ is given by
\begin{equation}\label{grad-b}
\nabla_\bb J(u_n)=
\bD(\hat{\bc})\left\{\bq - \gamma( u_n(1) -\beta) \b1 \right\}={\bf 0},
\end{equation}
where $\bD(\hat{\bc})$ is a diagonal matrix consisting of the linear parameters $(c_1,\ldots, c_n)$. 
\end{lemma}


\begin{proof} 
For each $j=0,1, \ldots, n$, we have
\begin{align*}
    \int_{b_j}^{b_{j+1}} a(x)(u^{\prime}_n(x))^2 \,dx &= \int_{b_j}^{b_{j+1}} a(x)\left(\sum_{i=0}^{j}c_iH(x - b_i)\right)^2 \,dx 
    =\left(\sum_{i=0}^{j}c_i\right)^2 \int_{b_j}^{b_{j+1}}a(x)dx,
\end{align*}
which yields
\[ 
     \dfrac{\partial}{\partial b_j}\int_{0}^{1}a(x)(u^{\prime}_n(x))^2dx = 
     a(b_j)\left(\sum_{i=0}^{j-1}c_i\right)^2 - a(b_j)\left(\sum_{i=0}^jc_i\right)^2
     = -2c_ja(b_j)\left(\sum_{i=0}^{j-1}c_i + \frac{c_j}{2}\right).
\] 
It is easy to see that
\[
\dfrac{\partial u_n(x)}{\partial b_j}=-c_jH(x - b_j) \quad\mbox{and}\quad \dfrac{\partial u_n(1)}{\partial b_j}=-c_j,
\]
which implies
\[
\dfrac{\partial}{\partial b_j}\int_{0}^{1}f(x)u_n(x)\,dx = -c_j\int_{b_j}^{1}f(x)\,dx \quad\mbox{and}\quad \frac{\partial}{\partial b_j} (u_n(1) - \beta)^2  = -2 c_j(u_n(1)- \beta).
\]
Hence, the optimality condition of \cref{min} with respect to $b_j$ is
\[
0=\dfrac{\partial}{\partial b_j} J(u_n)= -c_j\left[ a(b_j)\left(\sum_{i=0}^{j-1}c_i + \frac{c_j}{2}\right) - \int_{b_j}^{1}f(x)\,dx +\gamma (u_n(1)- \beta) \right]. 
\]
This completes the proof of the lemma.
\end{proof}

\begin{remark}\label{cl}
If $c_l=0$ for some $l\in\{1,\ldots,n\}$, then there is no $l^{th}$ equation of the optimality condition in \cref{grad-b}. In this case, the $l^{th}$ neuron has no contribution to the approximation $u_n(x)$ and hence can be removed or redistributed.
\end{remark}



Let $h_i=b_{i+1}-b_{i}$ for $i=0,\ldots,n$ and $h_{\text{min}}=\min\limits_{0\leq i\leq n} h_i$. The next lemma provides an upper bound for the condition number of the stiffness matrix $A(\bb)$.


\begin{lemma}\label{l:Cond}
Let $a(x)=1$ in \cref{pde}, then the condition number of the stiffness matrix $A(\bb)$ is bounded by $\mathcal{O}\left(n\,h^{-1}_{\textup{min}}\right)$.
\end{lemma}

\begin{proof}
For any vector $\bxi =(\xi_0,\ldots,\xi_n)^T\in \mathbb{R}^{n}$, denote its magnitude by 
$\big|\bxi\big|=\left(\sum\limits_{i=0}^n \xi^2_i\right)^{1/2}$. By the Cauchy-Schwarz inequality, we have

\begin{align}\nonumber
    \bxi^t A(\bb) \bxi 
    &= \int_0^1 \left( \sum_{i=0}^{n}\xi_iH(x-b_i) \right)^2 \, dx 
    \leq |\bxi|^2 \int_{0}^1 \left(\sum_{i=0}^n H(x-b_i)^2\right) \, dx  \\ \label{a} 
    &= |\bxi|^2\sum_{i=0}^{n}(1-b_i) < (n+1) |\bxi|^2.
\end{align}

To estimate the lower bound of the quadratic form $\bxi^t A(\bb) \bxi $, note that 
\begin{equation}\label{lower1}
    \bxi^t A(\bb) \bxi 
    = \sum_{j=0}^{n} \int_{b_j}^{b_{j+1}}\left( \sum_{i=0}^{j}\xi_iH(x-b_i) \right)^2dx
    =\sum_{j=0}^{n} h_{j}\left( \sum_{i=0}^{j}\xi_i \right)^2
    \geq  h_{\text{min}} \sum_{j=0}^{n} \left( \sum_{i=0}^{j}\xi_i \right)^2.
\end{equation}
Now, the lemma is a direct consequence of \cref{a}, \cref{lower1}, and the fact that 
\[
|\bxi|^2= \sum_{j=0}^n \biggl( \sum_{i=0}^j\xi_i-\sum_{i=0}^{j-1}\xi_{i}\biggr)^2
\leq 2 \sum_{j=0}^n \biggl( \sum_{i=0}^j\xi_i\biggr)^2 +2 \sum_{j=0}^n \biggl( \sum_{i=0}^{j-1}\xi_i\biggr)^2 \leq 4 \sum_{j=0}^n \biggl( \sum_{i=0}^j\xi_i\biggr)^2.
\]
This completes the proof of the lemma.
\end{proof}

\section{A Damped Block Newton Method}\label{Section4}

The optimality conditions in \cref{linear_eq} and \cref{grad-b} are systems of non-linear algebraic equations. One may use the Variable Projection (VarPro) method of Golub-Pereyra \cite{Golub1973} that substitutes the solution $\bc$ of \cref{linear_eq} into \cref{grad-b}
to produce a reduced non-linear system for the $\bb$. However, this approach can significantly complicate the non-linear structure of the system. 

In this section, we introduce a damped block Newton (dBN) method for solving the resulting non-convex minimization problem in \cref{min}.  The method employs a commonly used outer-inner iterative method by alternating between updates for $\bc$ and $\bb$. Per each outer iteration, the $\bc$ and the $\bb$ are updated by exact inversion and one step of a damped Newton method, respectively. To do so, we need to study the inversions of the stiffness and Hessian matrices.  

The stiffness matrix $A$ has the form of
\begin{eqnarray*}
 A(\bb)&=&\left(\begin{array}{ccccc}
 \int_{b_0}^{1}a(x)dx & \int_{b_1}^{1}a(x)dx & \int_{b_2}^{1}a(x)dx &\cdots    & \int_{b_{n}}^{1}a(x)dx \\[1mm]
 \int_{b_1}^{1}a(x)dx & \int_{b_1}^{1}a(x)dx & \int_{b_2}^{1}a(x)dx & \cdots    & \int_{b_n}^{1}a(x)dx\\[1mm]
  \int_{b_2}^{1}a(x)dx & \int_{b_2}^{1}a(x)dx & \int_{b_2}^{1}a(x)dx & \cdots    & \int_{b_n}^{1}a(x)dx\\[1mm]
 \vdots & \vdots & \vdots & \ddots & \vdots \\[1mm]
  \int_{b_n}^{1}a(x)dx & \int_{b_n}^{1}a(x)dx & \int_{b_n}^{1}a(x)dx & \cdots    & \int_{b_n}^{1}a(x)dx
\end{array}\right).
\end{eqnarray*}

\begin{lemma}\label{coefficientmatrixInverse}
The inversion of the stiffness matrix $A(\bb)$ is tri-diagonal given by 
    \begin{equation}\label{inverseA}
A(\bb)^{-1}=
\left(\begin{array}{ccccccc}
 \frac{1}{s_1} & -\frac{1}{s_1} & 0  &  0 & \cdots  & 0 & 0\\[1mm]
 -\frac{1}{s_1} & \frac{1}{s_1} + \frac{1}{s_2}   & -\frac{1}{s_2} & 0 &\cdots    & 0 & 0  \\[1mm]
 0 & -\frac{1}{s_2} & \frac{1}{s_2}+\frac{1}{s_3} & -\frac{1}{s_3} &  \cdots   & 0&0\\[1mm]
 \vdots & \vdots &  \vdots &  \vdots &\ddots  & \vdots &\vdots\\[1mm]
 0 & 0 & 0 & 0 &\cdots &\frac{1}{s_{n-1}}+ \frac{1}{s_n} & -\frac{1}{s_n}\\[1mm]
 0 & 0 & 0 & 0  & \cdots  &-\frac{1}{s_n}  & \frac{1}{s_n} + \frac{1}{s_{n+1}}
\end{array}\right),
\end{equation}
where $s_i= \int_{b_{i-1}}^{b_i}a(x)dx$ for $i = 1, \dots, n+1$.
Moreover, the inversion of the coefficient matrix $\mathcal{A}(\bb)$ is given by
\begin{equation}\label{CA-inv}
    {\mathcal A}(\bb)^{-1} =  A(\bb)^{-1} - \frac{\gamma A(\bb)^{-1} \bd \bd^{T}A(\bb)^{-1}}{1 + \gamma\bd^{T}A(\bb)^{-1}\bd}.
\end{equation}
\end{lemma}

\begin{proof}
It is easy to verify that $A(\bb)^{-1} A(\bb) = I$. \cref{CA-inv} is a direct consequence of \cref{inverseA} and the Sherman-Morrison formula.
\end{proof}


Now, we study how to solve the system of non-linear algebraic equations in \cref{grad-b}. As stated in \Cref{cl}, if $c_l=0$, then $b_l$ is removed from the system in \cref{grad-b}. Therefore, we may assume that $c_j\not=0$ for all $j\in \{1,\ldots,n\}$.
\begin{lemma}\label{hessian}
Assume that $c_j\not=0$ for all $j\in \{1,\ldots,n\}$ and that $a(x)$ is differentiable at $\{b_j\}_{j=1}^n$. Then we have
\begin{equation}\label{H}
   \nabla^2_{\bb}J(u_n)\equiv \mathcal{H}(\bc, \bb) =  \bD(\hat{\bc})\left(\bD(\bg)+ \gamma \hat{\bc}\,\b1^T\right),
\end{equation}
where $\bD(\bg)=\mbox{diag } \left(g_1,\ldots,g_n\right)$ is a diagonal matrix with the $j^{th}$ entry 
\[
g_j=\frac{\partial}{\partial b_j}q_j=-f(b_j) - a^\prime(b_j)u_n^\prime(b_j).
\]
\end{lemma}

\begin{proof}
Since $u^\prime_n(b_j)=\sum\limits_{i=0}^{j-1} c_i + c_{j}/2$ is independent of $b_j$, then we have 
\[
\dfrac{\partial}{\partial b_k}q_j =\left\{\begin{array}{ll} 
g_j, & k=j,\\[2mm]
0, & k\not= j
\end{array}\right. \quad\mbox{for}\quad k=1,\ldots,n.
\]
Now \cref{H} is a direct consequence of \cref{gradB=0} and the fact that $\nabla_{\bb} (u_n(1)-\beta)=-\bc$.
\end{proof}

If $a(x)$ is not differential at $b_l$ for some $l\in \{1,\ldots,n\}$, then $b_l$ lies at the physical interface. Hence, it should be fixed without further update. That is, it is removed from the system in \cref{grad-b}. 

If $g_l = 0$ for some $l\in \{1,\ldots,n\}$, then the matrix $\bD(\bg)$ is singular. By \cref{pde}, we have
\begin{equation}\label{g=0}
    0=g_l = -f(b_l)-a'(b_l)u_n'(b_l)\approx -f(b_l)-a'(b_l)u'(b_l) = a(b_l)u^{\prime\prime}(b_l). 
\end{equation}
This indicates that $b_l$ is an approximate inflection or undulation point of the solution $u(x)$. In the case of the former, $b_l$ may be removed or redistributed. Otherwise, $b_l$ should be fixed without further update. In both cases, $b_l$ is again removed from the system in \cref{grad-b}. To distinguish these two cases, notice that the linear parameter
    \[
    c_l = u_n^\prime(b_l^+)-u_n^\prime(b_l^-)
    \]
represents the change in slope of $u_n(x)$ at $x=b_l$. Hence, given a prescribed tolerance $\tau_1> 0$, if 
\begin{equation}\label{remove}
    |c_l| < \tau_1,
\end{equation}
  then $b_l$ corresponds to an approximate inflection point; otherwise, it corresponds to an approximate undulation point. 

\begin{lemma}\label{l:hessian}
Under the assumptions of \Cref{hessian}, if
$g_i \neq 0, c_i \neq 0$ for all $i\in\{1, \ldots, n\}$, and $\zeta\equiv 1-\gamma\sum\limits_{i=1}^nc_i/g_i \not= 0$, then the Hessian matrix $\mathcal{H}(\bc, \bb)$ is invertible. Moreover, its inverse is given by
 \begin{equation}\label{H-inv}
       \mathcal{H}(\bc, \bb)^{-1}=-\bD(\bg)^{-1}\left[ I +\dfrac{\gamma}{\zeta}\b1\hat{\bc}^T\bD(\bg)^{-1}\right] \bD(\hat{\bc})^{-1}. 
    \end{equation}
\end{lemma}

\begin{proof}
Clearly, by \cref{hessian}, the assumptions imply that $\mathcal{H}(\bc, \bb)$ is invertible. \cref{H-inv} is a direct consequence of the Sherman-Morrison formula.
\end{proof}

Now, we are ready to describe the dBN method (see \cref{alg:dBN} for a pseudocode). Given prescribed tolerances $\tau_1>0$ and $\tau_2 > 0$, let $\bb^{(k)}$ be the previous iterate, then the current iterate $\left(\bc^{(k+1)},\bb^{(k+1)}\right)$ is computed as follows:
\begin{itemize}
    \item[(i)] \textit{Compute the linear parameters} 
    \begin{equation*}
    \bc^{(k+1)} = {\cal A}\left(\bb^{(k)}\right)^{-1} {\cal F}\left(\bb^{(k)}\right)
    \end{equation*}

    \item[(ii)] \textit{Compute the search direction $\mathbf{p}^{(k)}=\left(p_1^{(k)},\dots,p_n^{(k)}\right)$.}
    Let $S=S_1\cup S_2$, where $S_1$ is the set of indices for the non-contributing neurons
    $$S_1=\left\{i\in\{1,\dots,n\}: \left\lvert c_i^{(k)}\right\rvert<\tau_1 \text{ or } b_i^{(k)}\notin I\right\}$$ 
    and $S_2$ is the set of indices for which the corresponding neurons are either redistributed or fixed 
    $$ S_2=\left\{i\in\{1,\dots,n\}: \left\lvert g_i^{(k)}\right\rvert<\tau_2 \text{ or DNE}\right\}.$$ 
    Compute
    \begin{equation}\label{directionVector}
p_i^{(k)}=\begin{cases}
        0, & i\in S,\\[2mm]
        \dfrac{1}{g_i^{(k)}}\left(\gamma\Bar{\beta}- q_i^{(k)}+ \gamma\Bar{\gamma}/\Bar{\zeta}\right), & i\notin S,
    \end{cases}
    \end{equation}
    where $\Bar{\beta}=\left(u_n(1)-\beta\right)$,  
    $\Bar{\zeta}=1-\gamma\sum\limits_{j\notin S}c_j^{(k)}\big/g_j^{(k)}$, and $\Bar{\gamma}=\!\sum\limits_{j\notin S} c^{(k)}_j\!\left(\!\gamma\Bar{\beta} \!-\!q_j^{(k)}\!\right)\big/g_j^{(k)}$.

    

    \item [(iii)] \textit{Calculate the stepsize} $\eta_k$ by performing a one-dimensional optimization
    \begin{equation*}
        \eta_k = \arg\!\!\min\limits_{\eta \in \R^+}J\left(u_{n}\left(x; \bc^{(k+1)}, \bb^{(k)} + \eta\mathbf{p}^{(k)}\right)\right).
    \end{equation*}
    \item [(iv)] \textit{Compute the non-linear parameters}
    \begin{equation*}
        \bb^{(k+1)} = \bb^{(k)} + \eta_k\mathbf{p}^{(k)}.
    \end{equation*}
    \item [(v)] \textit{Redistribute non-contributing breakpoints} $b_i^{(k+1)}$ for all $i \in S_1$ (e.g, per \cref{rmk_redist}) and sort $\bb^{(k+1)}$.   
\end{itemize}

\begin{remark}\label{rmk_redist}
    In the implementation of \cref{alg:dBN} in \Cref{Section numerical exp}, the particular redistribution for a neuron $b_l^{(k+1)}$ satisfying \cref{remove} is: 
    \[
    b_l^{(k+1)} \leftarrow\frac{b^{(k+1)}_{m-1}+b^{(k+1)}_{m}}{2},
    \]
    where $m\in\{1,\dots,n+1\}$ is an integer chosen uniformly at random.
\end{remark}



\begin{algorithm}
    \caption{A damped block Newton (dBN) method for \cref{min}}\label{alg:dBN}
    \begin{algorithmic}
    \REQUIRE{Initial network parameters $\bb^{(0)}$, coefficient function $a(x)$, the right-hand side function $f(x)$, boundary data $\alpha$ and $\beta$.}
    \ENSURE{Network parameters $\bc$, $\bb$}
    \FOR{$k=0,1\ldots$}
    \STATE $\triangleright$ \textit{Linear parameters}
    \STATE{$\bc^{(k+1)}\leftarrow {\cal A}\left(\bb^{(k)}\right)^{-1} {\cal F}\left(\bb^{(k)}\right)$} 
    \STATE $\triangleright$
    \textit{Non-linear parameters}
    \STATE{{\textit{Compute the search direction} }$\mathbf{p}^{(k)}$ \textit{as in}} \cref{directionVector}
    \STATE{$\eta_k \leftarrow \arg\!\!\min\limits_{\eta \in \R^+}J(u_n(x; \bc^{(k+1)}, \bb^{(k)} + \eta\mathbf{p}^{(k)}))$}
    \STATE{$\bb^{(k+1)} \leftarrow \bb^{(k)} + \eta_k\mathbf{p}^{(k)}$}
    \STATE $\triangleright$ \textit{Redistribute non-contributing neurons and sort $\bb^{(k+1)}$}
\ENDFOR
    \end{algorithmic}
\end{algorithm}


By \cref{CA-inv} and \cref{H-inv}, the linear parameters $\bc^{(k+1)}$ and the direction vector $\mathbf{p}^{(k)}$ may be calculated in $8(n+1)$ and $4(n+1)$ operations, respectively. Therefore, the computational cost of \cref{alg:dBN} per step is ${\cal O}(n)$. 
This is significantly less than Quasi-Newton approaches like BFGS, where the computational cost per iteration is ${\cal O}(n^2)$ (see \cite{NoceWrig06}, Chapter 6). 

\begin{remark}
    \label{initialize} The minimization problem in \cref{min} is non-convex and has many local and global minima. The desired one is obtained only if we start from a close enough first approximation. Because the non-linear parameters $\bb$ correspond to the breaking points that partition the interval $I=[0,1]$, as in {\em \cite{Cai2021linear, Liu_Cai_2022}}, in this paper we initialize the non-linear parameters $\bb$ to form a uniform partition of the interval $[0,1]$ and the linear parameters $\bc$ to be the corresponding solution on this uniform mesh. \end{remark}

\section{Adaptive dBN (AdBN) Method}\label{section adaptivity}

Starting with the uniform distribution of the breakpoints $\{b_i\}_{i=1}^n$ as an initial (see \cref{initialize}), the dBN efficiently moves them to a nearly optimal location (see \Cref{Section numerical exp}). However, the location of the resulting breakpoints may not be optimal for achieving the optimal convergence order of the shallow Ritz approximation in \cref{p:final_error_est}. This is because the uniform distribution is often not at the dBN basin of convergence when the number of breakpoints is relatively small. 

To circumvent this difficulty, we employ the adaptive neuron enhancement (ANE) method \cite{Liu_Cai_2022, Liu_Cai_Chen_2022, Cai2021DeepAdaptive} using the dBN as the non-linear solver per adaptive step, and the resulting method is referred to as the adaptive dBN (AdBN) method. The ANE integrates the ``moving mesh'' feature of the shallow Ritz method with the local mesh refinement of adaptive finite element method (aFEM). It starts with a relatively small NN and adaptively adds new neurons based on the previous approximation. Moreover, the newly added neurons are initialized at points where the previous approximation is not accurate. At each adaptive step, we use the dBN method to numerically solve the minimization problem in \cref{min}. 

A key component of the ANE is the marking strategy, defined by error indicators, which determines the number of new neurons to be added. Below, we describe the local indicators and the marking strategy used in this paper.

Letting $\mathcal{K} = [c, d] \subseteq [0, 1]$ be a subinterval, 
a modified local indicator of the ZZ type on $\mathcal{K}$ (see, e.g., \cite{CaZh:09}) is defined by
\[\xi_{\mathcal{K}} =  \lVert a^{-1/2}\left(G(a u_n')-a u_n'\right) \rVert_{L^2(\mathcal{K})},\]
where $G(a u_n')$ is the projection of $a u_n'$ onto the space of the continuous piecewise linear functions. The corresponding relative error estimator is defined by
\begin{equation}\label{relErrorEst}
    \xi =  \frac{\lVert a^{-1/2}\left(G(a u_n')-a u_n'\right) \rVert_{L^2(I)}}{\|u_n^{\prime}\|_{L^2(I)}}.
\end{equation}

For a given $u_n \in {\cal M}_n(I)$ with the breakpoints 
\[
0=b_0 <b_1< \ldots < b_n<b_{n+1}=1,
\]
let $\mathcal{K}^i = [b_{i},b_{i+1}]$, then ${\cal K}_n = \left\{ \mathcal{K}^i \right\}_{i = 0}^{n}$ is a partition of the interval $[0, 1]$.
Define a subset $\widehat{\mathcal{K}}_n\subset \mathcal{K}_n$ by using the following average marking strategy:
\begin{equation}\label{marking}
    \widehat{\mathcal{K}}_n=\left\{K\in\mathcal{K}_n: \xi_{\mathcal{K}}\geq\frac{1}{\#\mathcal{K}_n}\sum_{{\cal K}\in\mathcal{K}_n}\xi_{\mathcal{K}}\right\},
\end{equation}
where $\#\mathcal{K}_n$ is the number of elements in $\mathcal{K}_n$. 
The AdBN  method is described in \cref{alg:adBN}. 

\begin{algorithm}
    \caption{Adaptive damped block Newton (AdBN) method}\label{alg:adBN}
    \begin{algorithmic}
        \REQUIRE Initial number of neurons $n_0$, 
        parameters $a(x)$, $f(x)$, $\alpha$, and $\beta$, tolerance $\epsilon$,
        \STATE (1)  Compute an approximation to the solution $u_n$ of the optimization problem in \cref{min} by the dBN method; 
        \STATE (2) Compute the estimator $\xi_n=\left(\sum\limits_{K\in\mathcal{K}}\xi_K^2\right)^{1/2}/|u_n|_{H^1(I)}$;
        \STATE (3) If $\xi_n \leq \epsilon$, then stop; otherwise go to step (4);
        \STATE (4) Mark elements in $\widehat{\mathcal{K}}_n$ and denote by $\#\widehat{\mathcal{K}}_n$ the number of elements in $\widehat{\mathcal{K}}_n$;
        \STATE (5) Add $\#\widehat{\mathcal{K}}_n$ new neurons to the network and initialize them at the midpoints of elements in $\widehat{\mathcal{K}}_n$, then go to step (1)
    \end{algorithmic}
\end{algorithm}


For numerical experiments presented in the subsequent section, the dBN method in step (1) of \cref{alg:adBN} is stopped if the difference of the relative residuals for two consecutive iterates is less than a prescribed tolerance.

\section{Numerical Experiments} \label{Section numerical exp}

This section presents numerical results of the dBN and AdBN methods for solving \cref{pde}. In all the experiments, the parameters $\tau_1$ and $\tau_2$ were set to $10^{-10}$ and $10^{-6}$, respectively. The penalization parameter $\gamma$ was set to $10^4$ for the first two test problems (\cref{num ex 1} and \cref{num ex 2}) and $10^{13}$ for the third one (\cref{num ex 3}). For the AdBN method, a refinement occurred when the absolute difference of the relative residuals for two consecutive iterates was less than $10^{-3}$.

In order to evaluate the performance of the iterative solvers described in \cref{alg:dBN} and \cref{alg:adBN}, we compare the resulting approximations to the true solution. For each test problem, let $u$ and $u_n$ be the exact solution and its approximation in $\mathcal{M}_n(I)$, respectively. Denote the relative $H_1$ seminorm error by
\begin{equation*}
    e_n = \displaystyle \frac{|(u - u_n)|_{H^1(I)}}{|u|_{H^1(I)}}. 
\end{equation*}

Aside from observing the true error, we would like to evaluate the solver by estimating the order of convergence for select examples. Recall that according to \cref{p:final_error_est}, it is theoretically possible to achieve an order of convergence of ${\cal O}(n^{-1})$. However, since \cref{min} is a non-convex minimization problem, the existence of local minima makes it challenging to achieve this order. Therefore, given the neural network approximation $u_n$ to $u$ provided by the dBN method, assume that 
\begin{equation*}
    e_n = \left(\frac{1}{n}\right)^r,
\end{equation*}
for some $r > 0$. The larger this $r$ is, the better the approximation. In \cref{num ex 1}, we can see that $r$ improves when introducing adaptivity to dBN.

\subsection{Exponential Solution}\label{num ex 1}
The first test problem involves the function
\begin{equation}\label{Example1eq}
u(x) = x \left( \exp \left( -\frac{{(x - 1/3)^2}}{{0.01}} \right) - \exp \left( -\frac{{4}}{{9 \times 0.01}} \right) \right),
\end{equation}
serving as the exact solution of \cref{pde}. This example highlights the performance of dBN in comparison to current solvers and is used to motivate the addition of adaptivity.

We start by comparing the dBN method with a commonly used method: BFGS. In this comparison, we utilized a Python BFGS implementation from `scipy.optimize'. The initial network parameters for both algorithms were set to be the uniform mesh for $\bb^{(0)}$ with $\bc^{(0)}$ given by solving \cref{linear_eq}. We see in \cref{example1BFGSdBN} that dBN requires about 20 iterations to get an accuracy that BFGS cannot achieve after 250 iterations. Recall that the computational cost per iteration of dBN is ${\cal O}(n)$ while each iteration of BFGS has cost ${\cal O}(n^2)$. Not only does dBN decrease the relative error much more quickly than BFGS, but dBN also achieves a lower final error.

\begin{figure}[ht!]
    \centering
    
    \subfigure[Relative error $e_n$ vs number of iterations using 25 neurons, the ratio between the final errors is 0.753]
    {\includegraphics[width=0.45\textwidth]{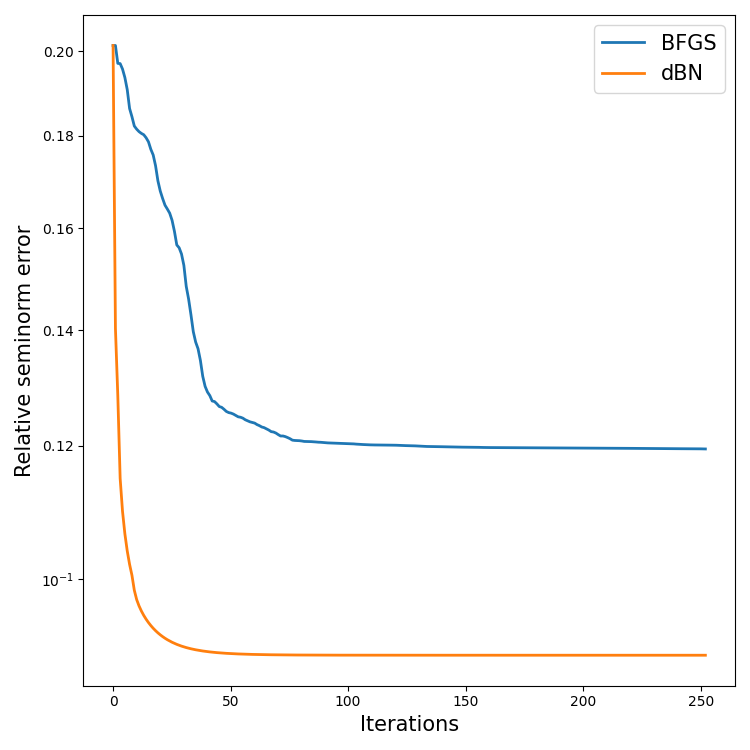}}
    \hspace{2em}
    \subfigure[Relative error $e_n$ vs number of iterations using 50 neurons, the ratio between the final errors is 0.794]{
        \includegraphics[width=0.45\textwidth]{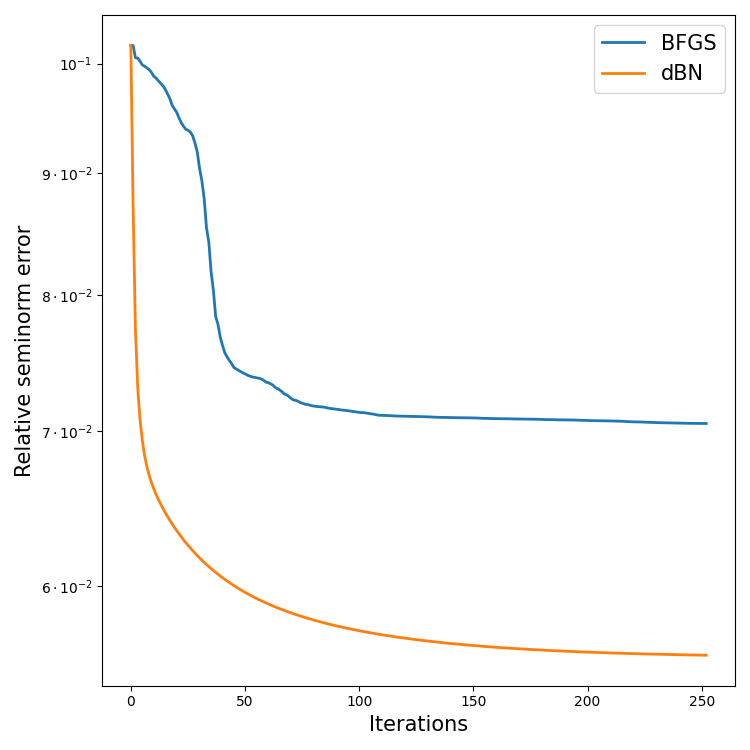}}
    \caption{Comparison between BFGS and dBN for approximating function \cref{Example1eq}}
    \label{example1BFGSdBN}
\end{figure}

Next we observe that our methods influence the movement of breakpoints, enhancing the overall approximation. In \cref{ex1Figure} (a), we present the initial neural network approximation of the exact solution in \cref{Example1eq}, obtained by using uniform breakpoints and determining the linear parameters through the solution of \cref{linear_eq}. The approximations generated by the dBN and AdBN methods are shown in \cref{ex1Figure} (b) and (c), respectively.

Worth noting is that the placement of breakpoints by AdBN appears to be more optimal than that achieved by dBN in terms of relative error. To verify this observation, we estimate the order of convergence for the approximation to \cref{Example1eq} in \cref{alg:dBN}. In \cref{tab1:errorvsN}, dBN  is applied for 1000 iterations for different values $n$, and the resulting $r$ is calculated. Since 
\[0.78<r<0.83,\]
it can be inferred that the algorithm gets stuck in a local minimum. Henceforth, to improve this order of convergence, we can use the AdBN method.

\begin{figure}[ht!]
  \centering 
  \subfigure[]{ 
    \includegraphics[width=0.31\linewidth]{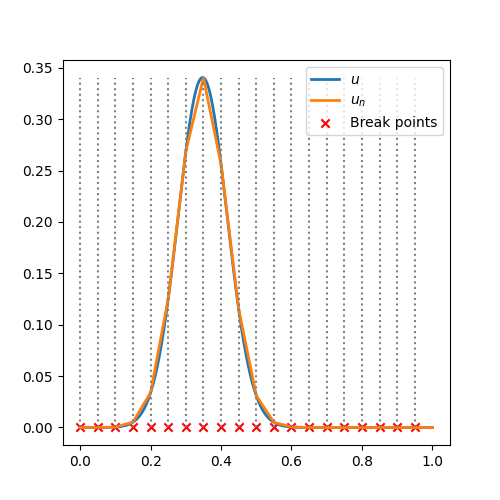}}
    \subfigure[]{ 
    \label{example1Opt} 
    \includegraphics[width=0.31\linewidth]{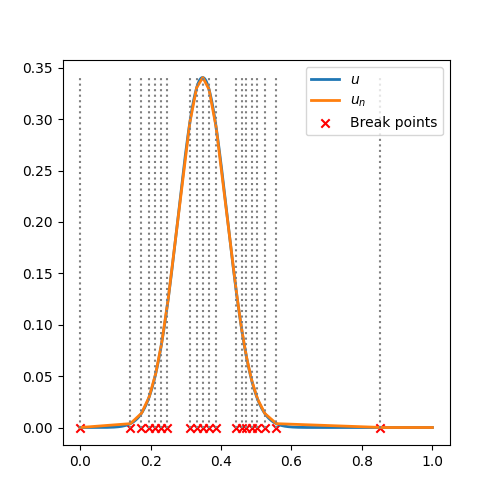}}
    \subfigure[]{\includegraphics[width=0.31\linewidth]{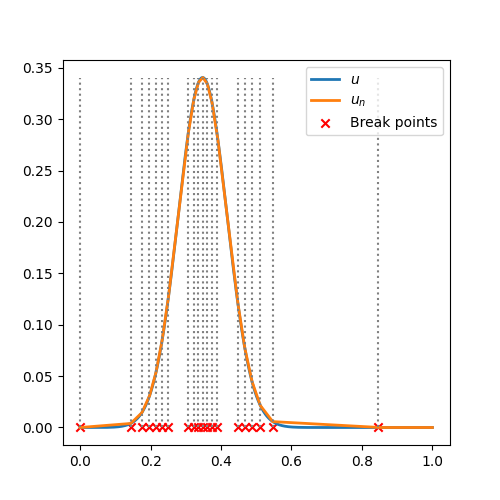}
    \label{example1Ref}
    }
    \caption{(a) initial approximation with $20$ uniform breakpoints, $e_n = 0.250$, (b) approximation after $500$ iterations,  $e_n = 0.104$, (c) adaptive approximation ($n=13,16, 20$), $e_n = 0.092$. 
    }
    \label{ex1Figure}
\end{figure}

\begin{table}[h!]
    \begin{center}
        \begin{tabular}{ |p{2cm}||p{2cm}|p{2cm}| }
        \hline
            Breakpoints&  $e_n$&  $ r$ \\
            \hline
            60 & $4.07 \times10^{-2}$ & 0.782\\
            90 & $2.88 \times10^{-2}$ & 0.789 \\
            120 & $1.92 \times10^{-2}$& 0.826\\
            150 & $1.89\times10^{-2}$ & 0.792\\
            180 & $1.61 \times10^{-2}$ & 0.796\\
            210 & $1.26 \times10^{-2}$ & 0.818 \\
            240 & $1.16 \times10^{-2}$& 0.814\\
            270 & $1.07\times10^{-2}$ & 0.811\\
            300 & $9.54 \times10^{-3}$ & 0.816\\
            330 & $8.94 \times10^{-3}$ & 0.813\\
            \hline
        \end{tabular}
        \caption{Relative errors $e_n$ and powers $r$ for different numbers of breakpoints after 1000 iterations.}
        \label{tab1:errorvsN}
    \end{center}
\end{table}

In fact, adding adaptivity improves the $r$ value. \cref{comparison_adapt} illustrates AdBN starting with 20 neurons, refining 8 times, and reaching a final count of 269 neurons. The stopping tolerance was set to $\epsilon = 0.01$. The recorded data in \cref{comparison_adapt} includes the relative seminorm error and the error estimator for each iteration of the adaptive process. Additionally, \cref{comparison_adapt} provides the results for dBN with a fixed 190 and 269 neurons after 300 iterations. Comparing these results to the adaptive run with the same number of neurons, we observe a significant improvement in rate, error estimator, and seminorm error within the adaptive run. The experiments confirm that AdBN improves the overall error. Furthermore, the order of convergence notably improves, especially with a larger number of neurons.

\begin{table}[h!]
    \begin{center}
        \begin{tabular}{ |p{3cm}||p{2cm}|p{1.5cm}|p{1.5cm}|}
        \hline
            NN (breakpoints)&  $e_n$&  $\xi_n$ & $r$\\
             \hline
            Adaptive (26) & $7.06 \times10^{-2}$ & 0.097 & 0.814\\ 
            Adaptive (33) & $5.68 \times10^{-2}$ & 0.073 & 0.820\\
            Adaptive (38) & $3.98 \times10^{-2}$& 0.049 & 0.838 \\ 
            Adaptive (67) & $2.87 \times10^{-2}$& 0.033 & 0.844 \\ 
            Adaptive (98) & $1.92 \times 10^{-2}$ & 0.021 & 0.862\\ 
            Adaptive (134) & $1.42\times10^{-2}$ & 0.015 & 0.868\\ 
            Adaptive (190) & $9.92\times10^{-3}$ & 0.011 & 0.879\\ 
            Adaptive (269) & $7.14\times10^{-3}$ & 0.007 & 0.883\\
            \hline
            Fixed (190) & $1.46\times10^{-2}$ & 0.016 & 0.806 \\
            Fixed (269) & $1.15\times10^{-2}$ & 0.012 & 0.798\\
             \hline
        \end{tabular}
        \caption{Comparison of an adaptive network with fixed networks for relative errors $e_n$, relative error estimators $\xi_n$, and powers $r$.}
        \label{comparison_adapt}
    \end{center}
\end{table}

\subsection{Non-smooth Solution}\label{num ex 2}
The exact solution of the second test problem is
\begin{equation}\label{Example2eq}
    u(x) = x^{2/3},
\end{equation}
noting that $u(x)\in H^{1+1/6-\epsilon}(I)$ for any $\epsilon >0$. As discussed in \Cref{Section2}, the order of approximation with $n$ uniform breakpoints is at most $1/6$, i.e., the error bound is ${\cal O}\left(n^{-1/6}\right)$. For $n=22$ breakpoints, \Cref{ex2Figure} depicts approximations on (a) the uniform mesh, (b) the moving mesh after $500$ iterations of the dBN, and (c) the adaptively refined moving mesh ($n=11, 16, 22$). The breakpoints move to the left where the solution exhibits the most curvature. As expected, the AdBN is more effective than the dBN. 


\begin{figure}[h!]
  \centering 
  \subfigure[]{ 
    \includegraphics[width=0.31\linewidth]{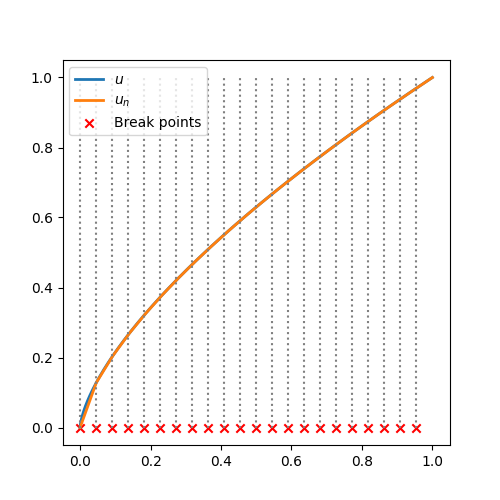}}
    \subfigure[]{ 
    \includegraphics[width=0.31\linewidth]{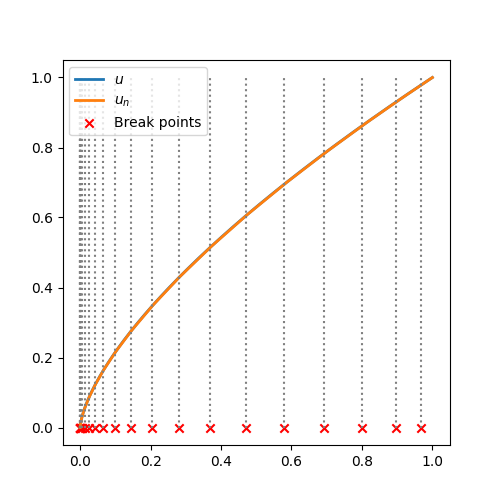}}
    \subfigure[]{\includegraphics[width=0.31\linewidth]{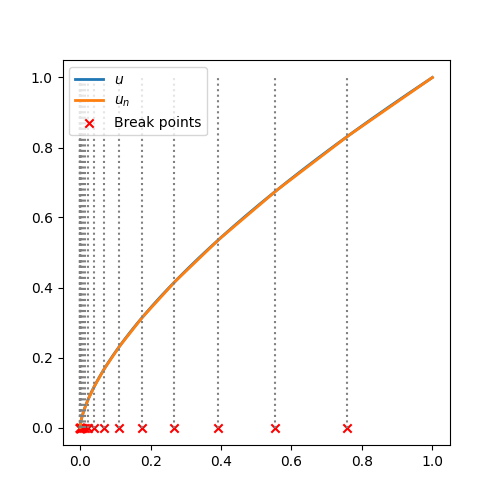}
    }
    \caption{(a) initial approximation with $22$ uniform breakpoints, $e_n = 0.300$, (b) approximation after $500$ iterations,  $e_n = 0.086$, (c) adaptive approximation ($n=11, 16, 22$), $e_n = 0.063$. 
    }
    \label{ex2Figure}
\end{figure}

\Cref{comparison_adapt2} compares dBN, AdBN, and aFEM. In this experiment, the dBN method was run for 300 iterations for 14, 18, and 23 breakpoints. Both AdBN and aFEM start with 10 breakpoints and are refined 4 and 16 times, respectively. The same marking strategy \cref{marking} is used to determine intervals for refinement for aFEM. The true error of the dBN using 23 breakpoints is slightly more accurate than that of aFEM after 11 refinements (45 breakpoints), and the error of AdBN after 4 refinements (31 breakpoints) is more accurate to that of aFEM after its $16^{th}$ refinement (122 breakpoints).



\begin{table}[h!]
    \begin{center}
        \begin{tabular}{ |p{4cm}||p{2cm}|p{1.5cm}|p{1.5cm}|}
        \hline
            Method (breakpoints)&  $e_n$& $\xi_n$&    $r$\\
             \hline 
            dBN (10) & $1.38 \times10^{-1}$&  0.163&0.859\\ 
            dBN (14) & $1.17 \times10^{-1}$ & 0.136&0.813\\ 
            dBN (18) & $1.01 \times10^{-1}$& 0.116&0.794 \\
            dBN (23) & $8.99 \times10^{-2}$& 0.077&0.768 \\
            \hline
            AdBN (14) & $9.22 \times10^{-2}$&  0.106&0.903\\ 
            AdBN (18) & $7.59 \times10^{-2}$ & 0.081&0.892\\ 
            AdBN (23) & $6.09 \times10^{-2}$& 0.063&0.892 \\
            AdBN (31) & $4.74 \times10^{-2}$& 0.048&0.888 \\

            \hline
            aFEM (10) & $3.39 \times 10^{-1}$ & 0.103&0.451\\ 
            aFEM (45) & $9.77\times10^{-2}$ & 0.033& 0.611\\
            aFEM (99) & $6.12\times10^{-2}$ & 0.020& 0.608\\
            aFEM (122) & $5.45\times10^{-2}$ & 0.017& 0.606\\
             \hline
        \end{tabular}
        \caption{Comparison of dBN, AdBN and aFEM for relative errors $e_n$, estimator errors $\xi_n$ and powers $r$. The initial number of breakpoints for AdBN and aFEM is 10. 
        }
        \label{comparison_adapt2}
    \end{center}
\end{table}

\subsection{Interface Problem}\label{num ex 3}
The third test problem is an elliptic interface problem whose solution has a discontinuous derivative. Specifically, the diffusion coefficient is $a(x) = 1 + (k-1)H(x-1/2)$, and the right-hand side and true solution are given respectively by 
\[
f(x) = \left\{\begin{array}{ll}
    8k(3x-1), & x \in (0, 1/2),\\[2mm]
        4k(k+1), & x \in (1/2, 1),\end{array}\right.
        \quad\mbox{and}\quad
u(x) = \left\{\begin{array}{ll}
    4kx^2(1-x), & x \in (0, 1/2),\\[2mm]
        [2(k+1)x - 1](1-x), & x \in (1/2, 1).\end{array}\right.
\]
This test problem was introduced in \cite{CAI2020109707} to show that the Bramble-Schatz least-squares NN method (i.e., the physics-informed neural network (PINN)) produces an unacceptable numerical approximation, that completely misses the interface physics (see Figure 4.4 (c) in \cite{CAI2020109707}), because the interface condition was not enforced. 

Notice that the diffusion coefficient $a(x)$ is piecewise constant with interface at $x=1/2$ and that its derivative is given by 
\[
a'(x) = \delta\left(x - 1/2\right) = \left\{\begin{array}{ll}
    +\infty, & x = 1/2,\\[2mm]
        0, & x \neq 1/2.\end{array}\right.
\]
As explained in \Cref{Section4}, if $b_l = 1/2$ for a $l \in \{1, \dots, n\}$, then the breakpoint $b_l$ remains unchanged.

The test problem employs $15$ neurons so that the uniformly distributed breakpoints do not contain the interface point $x=1/2$ upon initialization. For $k = 10^{6}$, this initial approximation is depicted in \cref{discont2} (a). 
The approximation after 100 iterations of dBN is depicted in \cref{discont2} (b) and is significantly more accurate because one of the breakpoints moved close to the interface point $x=1/2$. 
\cref{comparison_k} shows the relative errors of the shallow Ritz approximations on the uniform distribution and after 100 iterations of dBN for various values of $k$. 




    \begin{figure}[h!]
    \centering
    \subfigure[]
    {\includegraphics[width=0.45\textwidth]{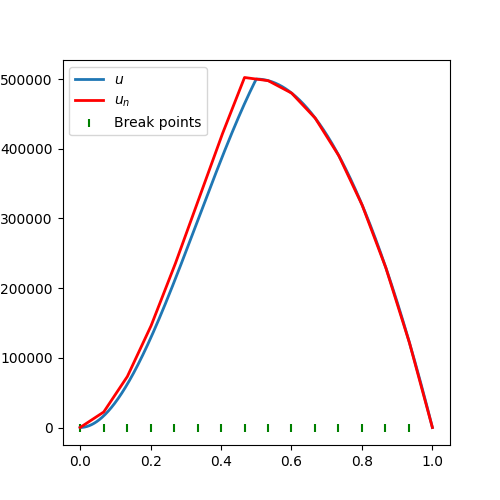}}
    \hspace{2em}
    \subfigure[]{
        \includegraphics[width=0.45\textwidth]{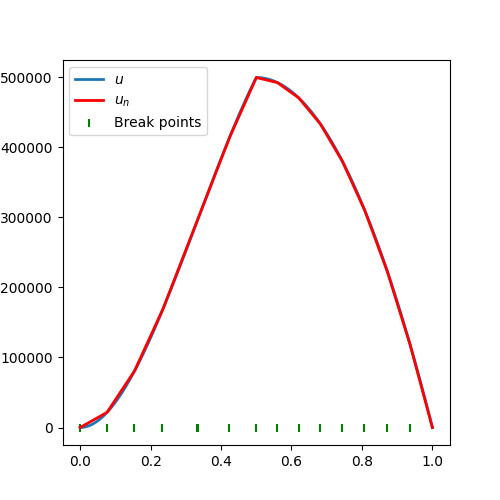}}
    \caption{For $k = 10^{6}$: (a) initial approximation with $15$ uniform breakpoints, $e_n = 0.204$ and (b) approximation after 100 iterations, $e_n =0.073$.}\label{discont2}
\end{figure}

\begin{table}[h!]
    \begin{center}
        \begin{tabular}{ |p{1cm}||p{2cm}|p{2cm}|}
        \hline
            $k$& $e_n$ (initial) & $e_n$ (dBN) \\
             \hline
            $10$&  $1.71\times 10^{-1}$&$6.86\times 10^{-2}$\\ $10^{2}$& $2.00\times10^{-1}$ & $7.06\times10^{-2}$ \\ 
          $10^{3}$&$2.03\times10^{-1}$  &$6.48\times10^{-2}$ \\ 
            $10^{4}$& $2.04\times10^{-1}$&$7.27\times10^{-2}$ \\ 
        $10^{5}$&$2.04\times10^{-1}$ &$7.28\times10^{-2}$\\
        $10^{6}$& $2.04\times 10^{-1}$& $7.30\times10^{-2}$\\      $10^{7}$&$2.04\times10^{-1}$ &$6.70\times10^{-2}$\\
      $10^{8}$&$2.04\times10^{-1}$ &$7.46\times10^{-2}$\\
             \hline
        \end{tabular}
        \caption{Relative energy error $e_n$ for $15$ neurons: initial approximation on uniform breakpoints and approximation after $100$ iterations of dBN.}
        \label{comparison_k}
    \end{center}
\end{table}

\section{Conclusion and Discussion}

The shallow Ritz method improves the order of approximation dramatically for one-dimensional non-smooth diffusion problems. However, determining optimal mesh locations (the non-linear parameters of a shallow NN) is a complicated, computationally intensive non-convex optimization problem. We have addressed this challenging problem by developing the dBN method, a specially designed non-linear algebraic solver that is capable of efficiently moving the uniformly distributed mesh points to nearly optimal locations. 

In the process of developing the dBN, we discovered that the exact inversion of the dense stiffness matrix is tri-diagonal; more importantly, we identified and overcame two difficulties of the Newton method for computing the non-linear parameters: (1) non-differentiable optimality condition and (2) vanishing linear parameters and diagonal elements of the Hessian matrix. The reduced non-linear system is differentiable and its Hessian is invertible. The computational cost of each dBN iteration is $\mathcal{O}(n)$.

For any given size $n$ of shallow NN, the uniformly distributed mesh points may not always be a good initial for the non-linear parameters and possibly prevent the dBN from moving mesh points to optimal locations. This was addressed by proposing the AdBN which combines the dBN with 
the adaptive neuron enhancement (ANE) method for adaptively adding neurons and is initialized at where the previous approximation is inaccurate. 

For non-smooth problems, the commonly used adaptive finite element method (aFEM) improves the inefficiency of the standard finite element method by adaptively refining the mesh, while the shallow Ritz method achieves the same goal by locating mesh points at optimal places. Both methods are non-linear, and their efficiencies depend on the local error indicator and the dBN, respectively. The AdBN is the combination of these two methods that efficiently moves and adaptively refines the mesh.

When using the shallow ReLU neural network, the condition number of the coefficient matrix for the diffusion problem is bounded by $\mathcal{O}\left(n\,h^{-1}_{\text{min}}\right)$ (see \cref{l:Cond}), the same as that obtained when using local hat basis functions. For applications to general elliptic partial differential equations and least-squares data fitting problems, the corresponding dBN method requires inversion of the mass matrix. However, the condition number of the mass matrix is extremely large. This difficulty will be addressed in a forthcoming paper.  


Extension of all components of the dBN to multi-dimension and to deep (two or more hidden layers) NNs  are not straightforward, and some of them may not be possible. For example, the exact inversion of the stiffness matrix in multi-dimension would not be sparse, and the Hessian would no longer be  diagonal or diagonal plus a low-rank update. Nevertheless, they do have some special structures that could be used for developing fast solvers (see, e.g., \cite{caiMass}). The most significant contribution of the dBN is how to handle singularities of the Hessian directly by using the physical meaning of the nonlinear parameters. This idea is valid for multi-dimension. For example, the structure-guided Gauss-Newton (SgGN) method was developed in \cite{SgGN} for solving shallow NN least-squares approximation in multi-dimension, and outperforms the commonly used Levenberg–Marquardt (LM) algorithm \cite{Levenberg, Marquardt} by a large margin. The LM deals with singularities of the GN matrix by adding a regularization and hence changes the GN search direction. In summary, the methodology presented in this paper has a great potential to be extended to multi-dimension and deep NNs for achieving better efficiency, accuracy, and reliability than the commonly used {\it generic} training algorithms in machine learning.

\smallskip

\appendix
\section{Enforcing the Dirichlet Boundary Condition Algebraically}\label{Appendix}
Another way to make a function $u_n \in {\cal M}_n(I) \cap \{u_n(0) = \alpha\}$ satisfy the Dirichlet boundary condition $u_n(1) = \beta$ is by enforcing this algebraically. Consider the energy functional given by 
\begin{equation*}
    {\cal J}(v) = \frac{1}{2}\int_0^1a(x)(v^{\prime}(x))^2dx - \int_{0}^1f(x)v(x)dx.
\end{equation*}
Let $\bd= \mathbf{\Sigma}(1)$ and consider the Lagrangian function 
\begin{align*}
    {\cal L}(\bc, \bb, \lambda) &= {\cal J}(u_n(x; \bc, \bb)) + \lambda(u_n(1) - \beta)\\
    & = {\cal J}(u_n(x; \bc, \bb)) + \lambda(\bd^{T}\bc + \alpha - \beta),
\end{align*}
hence, if $u_n$ minimizes ${\cal J}(v)$ for $v \in {\cal M}_n(I) \cap \{u_n(0) = \alpha\}$, subject to the constraint $u_n(1) = \beta$, then by the KKT conditions:
\begin{equation}\label{cps2}
    \nabla_{\bc} {\cal L}\left(u_n\right)={\bf 0}, \quad \frac{\partial}{\partial \lambda}{\cal L}\left(u_n\right) = 0
    \quad\mbox{and}\quad
    \nabla_{\bb} {\cal L}\left(u_n\right)={\bf 0}.
\end{equation}
The first two equations in \cref{cps2} can be written as 
\begin{equation}\label{newLinearSystem}
   \left(
    \begin{array}{cc}
    A(\bb)& \bd \\
    \bd^{T} & 0
    \end{array}
    \right) \left(
    \begin{array}{c}
    \bc  \\
    \lambda 
    \end{array}
    \right) = \left(
    \begin{array}{c}
    \bff(\bb)  \\
    \beta - \alpha 
    \end{array}
    \right),
\end{equation}
which can be solved efficiently, by writing the matrix on the left-hand side as 
\begin{equation*}\label{lagrange1}
    \left(
    \begin{array}{cc}
    A(\bb)& \bd \\
    \bd^{T} & 0
    \end{array}
    \right) = \left(
    \begin{array}{cc}
    I& 0 \\
    \bd^{T}A(\bb)^{-1} & 1
    \end{array}
    \right)\left(
    \begin{array}{cc}
    A(\bb)& \bd \\
    0 & -\bd^{T}A(\bb)^{-1}\bd
    \end{array}
    \right), 
\end{equation*}
and since
\begin{equation*}
    \left(
    \begin{array}{cc}
    I& 0 \\
    \bd^{T}A(\bb)^{-1} & 1
    \end{array}
    \right)^{-1} = \left(
    \begin{array}{cc}
    I& 0 \\
    -\bd^{T}A(\bb)^{-1} & 1
    \end{array}
    \right),
\end{equation*}
it follows that \cref{newLinearSystem} is equivalent to 
\begin{equation*}
    \left(
    \begin{array}{cc}
    A(\bb)& \bd \\
    0 & -\bd^{T}A(\bb)^{-1}\bd
    \end{array}
    \right)\left(
    \begin{array}{c}
    \bc  \\
    \lambda 
    \end{array}
    \right) = \left(
    \begin{array}{cc}
    I& 0 \\
    -\bd^{T}A(\bb)^{-1} & 1
    \end{array}
    \right) \left(
    \begin{array}{c}
    \bff(\bb)  \\
    \beta - \alpha 
    \end{array}
    \right)
\end{equation*}
which can be solved by finding $\lambda$ first and then solving for $\bc$. According to \cref{inverseA}, the computational cost is ${\cal O}(n)$. 

On the other hand, for the non-linear parameters $\bb$, the Hessian matrix $\nabla_{\bb}^{2}{\cal L}(u_n) =  \bD(\hat{\bc})\bD(\bg)$, where $\bD(\bg)$ is defined in \cref{H}. Therefore, as described in \Cref{Section4}, we solve \cref{pde} iteratively, alternating between solving exactly for $(\bc^{T}, \lambda)^T$ and performing a Newton step for $\bb$.  

\section*{Code Availability Statement}The code used to generate the findings of this study is openly available in GitHub at \url{https://doi.org/10.5281/zenodo.17336056}.

\bibliographystyle{plain}
\bibliography{ref}

\end{document}

%% file: ex_shared_v2.tex
\usepackage{lipsum}
\usepackage{amsfonts}
\usepackage{graphicx}
\usepackage{epstopdf}
\usepackage{algorithmic}
\ifpdf
  \DeclareGraphicsExtensions{.eps,.pdf,.png,.jpg}
\else
  \DeclareGraphicsExtensions{.eps}
\fi

\newsiamremark{hypothesis}{Hypothesis}
\crefname{hypothesis}{Hypothesis}{Hypotheses}
\newsiamthm{claim}{Claim}

\title{Efficient Shallow Ritz Method for 1D Diffusion Problems
\thanks{This work was supported in part by the National Science Foundation under grant DMS-2110571 and by the Department of Energy under NO.B665416.  This work was performed under the auspices of the U.S. Department of Energy by Lawrence Livermore National Laboratory under Contract DE-AC52-07NA27344 (LLNL-JRNL-862433 and No. B665416).}}

\author{Zhiqiang Cai\thanks{Department of Mathematics, Purdue University, West Lafayette, IN
    (\email{caiz@purdue.edu},  \email{adoktoro@purdue.edu}, \email{herre125@purdue.edu}).}
\and Anastassia Doktorova\footnotemark[2]
\and Robert D. Falgout\thanks{Lawrence Livermore National Laboratory, Livermore, CA (\email{rfalgout@llnl.gov})}
\and César Herrera\footnotemark[2]
}

\usepackage{amsopn}